%% file: implicit_operations_v1.tex
\documentclass[a4paper,10pt,reqno]{amsart}
\usepackage{amsthm,amsmath,amssymb,amsfonts}
\usepackage[T1]{fontenc}
\usepackage[utf8]{inputenc}
\usepackage{graphicx}
\usepackage{bbold}
\usepackage{xcolor}
\usepackage{caption}
\usepackage{enumitem}
\usepackage{mathtools}
\usepackage{microtype}
\usepackage{stmaryrd}
\usepackage{tikz-cd}
\usepackage{hyperref}
\hypersetup{
    colorlinks = true,
    linkbordercolor = {red},
   	linkcolor ={magenta},
	anchorcolor = {pink},
	citecolor =  {magenta},
	filecolor = {blue},
	menucolor = {blue},
	runcolor =  {blue},
	urlcolor = {magenta},
}

\newtheorem{theorem}{Theorem}[section]
\newtheorem{lemma}[theorem]{Lemma}

\newtheorem{corollary}[theorem]{Corollary}
\newtheorem{claim}[theorem]{Claim}

\theoremstyle{definition}
\newtheorem{definition}[theorem]{Definition}

\theoremstyle{remark}
\newtheorem{remark}[theorem]{Remark}

\input{macros.tex}

\title{On implicit operations and balanced categories}

\author{Luca Reggio}
\address{
Dipartimento di Matematica \textnormal{\emph{Federigo Enriques}}\newline 
Universit\`a degli Studi di Milano\newline 
Milan, Italy
}
\email{luca.reggio@unimi.it}

\begin{document}

\begin{abstract}
The problem of determining whether a category is balanced (i.e., every morphism that is both an epimorphism and a monomorphism is an isomorphism) has been widely investigated, especially in algebra.
For complete categories with an extremal projective generator, we characterise the property of being balanced, and the stronger property that every monomorphism is regular, in terms of implicit operations. These characterisations apply, for example, to any quasiequational class of algebras, possibly with a proper class of function symbols and unbounded arities, that admits free algebras.
\end{abstract}

\maketitle

\section{Introduction}\label{s:introduction}
A category is \emph{balanced} if every morphism that is both an epimorphism and a monomorphism is an isomorphism; equivalently, if every epimorphism is \emph{extremal} (see Section~\ref{s:category-theory} for a definition). We shall refer to the latter as the \emph{\ES~property} (``Epimorphisms are Surjective''), because in many concrete categories, including all quasivarieties of algebras, it coincides with surjectivity of epimorphisms. The problem of determining whether a category is balanced, and more generally of describing the images of epimorphisms, is often difficult. For example, the surjectivity of epimorphisms in the category of planar graphs and edge-preserving functions is essentially equivalent to the Four Colour Theorem \cite{Fawcett1986}. For classes of algebras, this problem has been widely studied, particularly via Isbell's concept of \emph{dominion}. See e.g.\ \cite{Isbell1966,HI1967,Storrer1968} for earlier work on this topic.

Another approach is via implicit operations, which we will now review briefly.
Consider a class of algebras $\K$ in an algebraic signature $\tau$,\footnote{Here, we assume that $\tau$ is a small set of function symbols, each of which has finite arity. These assumptions will be dropped later on.} and a set $I$. An \emph{implicit total operation of arity $I$} in $\K$ is a class $f\coloneqq \{f_{A}\mid A\in \K\}$ consisting of functions $f_{A}\colon A^{I}\to A$ that are preserved by all homomorphisms in $\K$, meaning that
\[
h(f_{A}(a_{i})_{i\in I}) = f_{B}(h(a_{i})_{i\in I})
\]
for every homomorphism of $\K$-algebras $h\colon A\to B$ and every element $(a_{i})_{i\in I}\in A^{I}$. Every $\tau$-term in $n$ variables induces an $n$-ary implicit operation in $\K$, but there may exist implicit operations that are not induced by any term. 

The study of implicit operations has considerably benefited from the categorical approach to universal algebra pioneered by Lawvere and Linton \cite{Lawvere1963,Linton1966}. In fact, an implicit operation of arity $I$ is precisely a natural transformation $f\colon U^{I}\to U$, where $U\colon \K\to\Set$ is the underlying-set functor. 
If $\K$ admits free algebras, $U$ is represented by the free $\K$-algebra on one generator, and the natural transformations $U^{I}\to U$ are in bijective correspondence with the elements of the free $\K$-algebra on the set $I$. Therefore, in this case, every implicit operation in $\K$ is induced by a term.
This functorial viewpoint has been exploited to investigate, e.g., pseudovarieties of finite algebras \cite{RP1981,Reiterman1982}, equational completions of copresheaves~\cite{KG1971}, and the definability of implicit operations \cite{Isbell1973,Hodges1980,Isbell1983,Hebert1990}.

To provide a first criterion for determining when \ES~holds in the category of algebras for a possibly infinitary algebraic theory, Linton employed implicit \emph{partial} operations \cite[Proposition~7]{Linton1966}. An \emph{implicit partial operation of arity~$I$} in~$\K$, written $U^I \rightsquigarrow U$, is a natural transformation $D\to U$, with $D$ a subfunctor of $U^{I}$. 
Later results considered a strengthening of \ES, namely \SES~(the \emph{strong} \ES~property), stating that every monomorphism is regular. 
Consider an implicit partial operation~$f$ in~$\K$ that depends on finitely many variables $\overline{x}$. We say that $f$ is \emph{definable} by a primitive positive formula, i.e.\ one of the form $\Psi(\overline{x},y)\coloneqq\exists z_{1}\cdots\exists z_{k}\Phi$ with $\Phi$ a conjunction of $\tau$-equations, if every algebra $A$ in $\K$ satisfies
\[
\forall \overline{x}\, \forall y \, (f(\overline{x})=y \leftrightarrow \Psi(\overline{x},y))
\]
when $f$ is interpreted as $f_{A}$.
Bacsich showed that if $\K$ is a quasivariety of algebras, or more generally a Horn class of first-order structures,\footnote{A \emph{Horn class} is the class of models of a universal Horn theory. By Maltsev's theorem, they are the elementary classes closed under substructures and Cartesian products \cite{Maltsev1966}. If the signature does not contain relation symbols, Horn classes coincide with quasivarieties of algebras.} then $\K$ satisfies \SES~if, and only if, every implicit partial operation definable by a primitive positive formula is induced by a term \cite[Theorem~4]{Bacsich74}. 

In the 1990s, H{\'e}bert improved upon Bacsich's result by showing that a Horn class satisfies properties \ES~and \SES~precisely when certain implicit partial operations can be extended to implicit total operations \cite[Théorèmes~3.12, 3.13]{Hebert1993}. 
In particular, a quasivariety of algebras satisfies \ES~(respectively, \SES) if, and only if, every implicit partial operation whose domain $D$ preserves limits (respectively, products) can be explicitly defined by a term. 
He then extended these characterisations to classes of models of \emph{limit-theories} in~\cite{Hebert1998}, which are strictly more general than Horn classes and coincide with \emph{locally presentable categories} (cf.\ e.g.\ \cite{GabrielUlmer1971,ar94book}).

Call an implicit partial operation \emph{continuous} if its domain $D$ preserves limits, and \emph{product-preserving} if $D$ preserves products.
In \cite[Theorem~9]{Hebert1998}, H{\'e}bert proved the following characterisations for any category~$\M$ of models of a limit-theory, with morphisms all homomorphisms:
\begin{itemize}
\item $\M$ is balanced if, and only if, every continuous implicit partial operation has ``subobject-closed'' domain (see \cite[Definition~4]{Hebert1998} for a definition).
\item All extremal monos in $\M$ are regular if, and only if, every product-preserving implicit partial operation extends to a continuous implicit partial operation.
\item All extremal epis in $\M$ are surjective if, and only if, every continuous implicit partial operation with subobject-closed domain extends to a total one.
\end{itemize}
Let us mention that, in \emph{op.~cit.}, H{\'e}bert also described the relation between dominions and implicit partial operations at this level of generality; see \cite[Theorem~8]{Hebert1998}.

\medskip
The main result of this paper (Theorem~\ref{t:main}) is a characterisation of properties \ES~and \SES~in terms of implicit operations, which applies to any complete category with a ``good generator'' (namely, an extremal projective generator, see Definition~\ref{def:extremal-projective-generator}). There are two main differences between H{\'e}bert's framework and ours. First, the class of categories we consider has an algebraic flavour, since the existence of a good generator generalises the existence of free algebras. Locally presentable categories need not admit such a generator. On the other hand, a category satisfying our assumptions need not be locally presentable. Second, H{\'e}bert's aforementioned results rely on a concrete presentation of a locally presentable category as a category of models of a limit-theory. In particular, the implicit operations refer to the underlying-set functor of the category of models (or a variation thereof, in the case of multi-sorted limit-theories). By contrast, we work with the representable functor associated with the generator, and our proofs are purely categorical.

Let us say that an implicit partial operation is \emph{representable} (respectively, \emph{principal}) if its domain is a representable functor (respectively, a principal subfunctor, see Definition~\ref{def:principal-subfunctor}). Moreover, recall that a category is \emph{well-powered} (respectively, \emph{well-copowered}) if, up to isomorphism, every object has a small set of subobjects (respectively, of epi-quotients).
Our main result can be stated as follows.

\begin{theorem}\label{t:main}
Let $\cat$ be a complete category admitting an extremal projective generator $G$, and let $V\coloneqq \cat(G,-)\colon \cat\to\Set$. The following statements hold:
\begin{enumerate}[label=(\roman*)]
\item\label{i:ES-char-representable} $\cat$ satisfies \ES~if, and only if, every representable implicit partial operation $V^{I} \rightsquigarrow V$ can be extended to a total one.
\item\label{i:SES-char} $\cat$ satisfies \SES~if, and only if, every principal implicit partial operation $V^{I} \rightsquigarrow V$ can be extended to a total one.
\end{enumerate}
If, in addition, $\cat$ is well-copowered, then the following statement holds:
\begin{enumerate}[label=(\roman*),resume]
\item\label{i:ES-char} $\cat$ satisfies \ES~if, and only if, every continuous implicit partial operation $V^{I} \rightsquigarrow V$ can be extended to a total one.
\end{enumerate}
\end{theorem}

It is worth noting that any well-copowered category that satisfies the assumptions of Theorem~\ref{t:main} is cocomplete (see Remark~\ref{r:cocompleteness-from-gen}).
Furthermore, while item~\ref{i:ES-char} in Theorem~\ref{t:main} is a direct generalisation of the aforementioned result for quasivarieties of algebras, item~\ref{i:SES-char} concerns principal implicit partial operations rather than product-preserving ones. Principal implicit partial operations are product-preserving, but the converse does not hold in general.

\medskip
We shall now give a few examples of categories that satisfy the assumptions of Theorem~\ref{t:main}. 
Consider any category monadic over $\Set$. Up to equivalence, these are precisely the equational classes of algebras, possibly with a proper class of function symbols and unbounded arities, that admit free algebras. Cf.\ \cite[Proposition~I.3.8]{Johnstone1982}. 
Their (regular epi)-reflective subcategories correspond to the quasiequational classes of algebras that admit free algebras, again with the same proviso concerning possibly large sets of function symbols and unbounded arities. An example of equational class of algebras that does not admit free algebras is the class of complete Boolean algebras~\cite{Gaifman1964,Hales1964}.

We refer to (quasi)equational classes of algebras as \emph{generalised (quasi)varieties}. (This is a slight abuse of language, because Birkhoff's Variety Theorem holds for large signatures if, and only if, the universe of small sets is not measurable~\cite{AT2005}.)
Any generalised (quasi)variety of algebras that admits free algebras is complete and has an extremal projective generator, namely the free algebra on one generator.
Therefore, it satisfies items~\ref{i:ES-char-representable}--\ref{i:SES-char} of Theorem~\ref{t:main}. This applies, e.g., to the category

\begin{itemize}
	\item[$\Frm$:] frames and frame homomorphisms.  This category is monadic over $\Set$, but it is not well-copowered (see e.g.\ \cite[\S IV.6.6]{PP2012}).
\end{itemize}

If there is a small set of function symbols, then the supremum of their cardinalities exists, and the generalised (quasi)variety is locally presentable, hence well-copowered \cite[Theorem~1.58]{ar94book}. The following are examples of well-copowered categories that satisfy the assumptions of Theorem~\ref{t:main} but are not locally presentable.\footnote{A useful criterion for proving that a category is \emph{not} locally presentable is as follows: if both $\cat$ and $\cat^{\op}$ are locally presentable, then $\cat$ is equivalent to a poset. See \cite[Satz~7.13]{GabrielUlmer1971} or \cite[Theorem~1.64]{ar94book}.}

\begin{itemize}
\item[$\CABA$:] complete atomic Boolean algebras and complete homomorphisms. It is dual to $\Set$, which is cocomplete, well-powered, and admits an extremal injective cogenerator (the dual of an extremal projective generator), namely $\{0,1\}$.
\item[$\CH$:] compact Hausdorff spaces and continuous maps. It is monadic over $\Set$, see \cite{Manes1976}, and it is well-copowered because epimorphisms are surjective.
\item[$\Stone$:] Stone spaces ($0$-dimensional compact Hausdorff spaces) and continuous maps. It is dual to the category $\BA$ of Boolean algebras and their homomorphisms~\cite{Stone1936}, which satisfies the dual properties. In particular, the two-element Boolean algebra is an extremal injective cogenerator in $\BA$.
\item[$\Top^\op$:] the opposite of the category of topological spaces and continuous maps. It is equivalent to a generalised quasivariety of algebras~\cite{BP1995}, and it is well-copowered because $\Top$ is well-powered.
\item[$\CSLat$:] complete join-semilattices and complete homomorphisms. It is monadic over $\Set$ (in fact, it is the category of algebras for the covariant power-set monad), and it is well-copowered because it is well-powered and self-dual.
\end{itemize}

In the examples above, the relevant implicit operations have arbitrary arities. However, for \emph{locally finitely presentable} categories with the amalgamation property and an extremal projective generator, properties \ES~and \SES~are equivalent, and can be characterised in terms of \emph{finitary} implicit operations (Theorems~\ref{t:WES-to-SES} and~\ref{t:SES-vs-implicit-operations-with-amalgam}).

\medskip
Beyond detecting when \ES~or \SES~holds, implicit operations have also been used to ``repair'' its failure. Given a quasivariety that does not satisfy \SES, one can seek a companion quasivariety that does, called a \emph{Beth companion}~\cite{CKM2026}. A categorical approach to Beth companions of locally finitely presentable categories, enforcing either \ES~or \SES, is developed in~\cite{DLR2026}.

\subsubsection*{Structure of the paper} In Section~\ref{s:preliminaries}, we outline the necessary preliminaries concerning category theory, and the properties \ES~and \SES. In Section~\ref{s:implicit-operations-defined}, we define the notion of implicit partial operations that will be employed throughout. Items~\ref{i:ES-char-representable} and~\ref{i:ES-char} in Theorem~\ref{t:main} are proved in Section~\ref{s:implicit-ES}, while item~\ref{i:SES-char} is proved in Section~\ref{s:implicit-SES}. In Section~\ref{s:lfps-with-AP}, we study locally finitely presentable categories with the amalgamation property, in order to establish a reduction to finitary implicit operations in Section~\ref{s:finitary-implicit-partial-operations}.

\section{Preliminaries}\label{s:preliminaries}

\subsection{Category theory}\label{s:category-theory}
All categories under consideration are assumed to be locally small.
A category $\cat$ is \emph{well-powered} if, for every object $X$ of $\cat$, the collection of subobjects of $X$ is a set (as opposed to a proper class). Moreover, $\cat$ is \emph{well-copowered} if $\cat^{\op}$ is well-powered.

Recall that an epimorphism $e$ is an \emph{extremal epimorphism} provided that, for any decomposition $e= m\circ f$, if $m$ is monic then it is an isomorphism. \emph{Extremal monomorphisms} are defined dually. We shall consider the following notion of a generator (sometimes called a ``strong generator'', cf.\ e.g.\ \cite{ar94book}).

\begin{definition}
An object $G$ of a category $\cat$ is an \emph{extremal generator} if it admits small copowers (that is, for every set $I$, the $I$-fold copower $\coprod_I{G}$ of $G$ exists in $\cat$) and, for every object $X$ of $\cat$, the canonical arrow
\[
\coprod_{\cat(G,X)}{G} \to X
\]
is an extremal epimorphism.
\end{definition}

If $G$ is an extremal generator, the following property holds:
\begin{itemize}
\item for each object $X$, and each proper subobject $m\colon S\emb X$, there exists an arrow $G\to X$ that does not factor through $m$.
\end{itemize}
In fact, assuming that $G$ admits set-indexed copowers, the latter property is equivalent to saying that $G$ is an extremal generator. 

\begin{definition}
An object $A$ of a category $\cat$ is \emph{extremal projective} if the functor $\cat(A,-)\colon \cat \to \Set$ preserves extremal epimorphisms. 
\end{definition}

Because extremal epimorphisms in $\Set$ are precisely the surjections, $A$ is extremal projective precisely when it satisfies the following lifting property: for every extremal epimorphism $e\colon Y\to X$, and every arrow $f\colon A\to X$, there exists an arrow $g\colon A\to Y$ such that $f = e\circ g$.

\begin{definition}\label{def:extremal-projective-generator}
An object of a category is an \emph{extremal projective generator} if it is both extremal projective and an extremal generator.
\end{definition}

\begin{lemma}\label{l:basic-props-complete-gen}
The following statements hold in any complete category $\cat$ that admits an extremal generator:
\begin{enumerate}[label=(\roman*)]
\item\label{i:well-powered} $\cat$ is well-powered;
\item\label{i:factorisations} $\cat$ has (extremal epi, mono) factorisations.
\end{enumerate}
\end{lemma}
\begin{proof}
Item~\ref{i:well-powered} is a particular case of a more general fact: any category with finite limits and a \emph{strong generating set} is well-powered. See e.g.\ \cite[\S A1.4, Remark~1.4.17]{Elephant1} or \cite[Proposition~4.5.15]{Borceux1}.
Item~\ref{i:factorisations} holds for any well-powered complete category; see e.g.\ \cite[Propositions~4.3.7, 4.4.2 and~4.4.3]{Borceux1}.
\end{proof}

\begin{remark}\label{r:cocompleteness-from-gen}
If $\cat$ is a well-copowered complete category admitting an extremal generator, then it is cocomplete. Just observe that $\cat$ is well-powered by Lemma~\ref{l:basic-props-complete-gen}\ref{i:well-powered}, and every complete, well-powered and well-copowered category that admits an extremal generator is cocomplete by \cite[Exercise~12H]{ahs06book}.
\end{remark}

For any category $\cat$, write $\CP{\cat}$ for the category of \emph{small copresheaves} on~$\cat$. Its objects are the functors $\cat\to \Set$ that are small colimits of representable functors, and its arrows are the natural transformations. The category $\CP{\cat}$ is indeed locally small, and the Yoneda embedding
\[
\yo\colon \cat^{\op}\to \CP{\cat}, \ \ A\mapsto \yo(A)\coloneqq \cat(A,-)
\]
embeds $\cat$ as a full subcategory of $\CP{\cat}$.

\subsection{Balancedness-type properties}\label{s:definability-props}
It is a simple observation that the following conditions are equivalent for any category $\cat$:
\begin{enumerate}[label=(\arabic*)]
\item $\cat$ is balanced;
\item every monomorphism in $\cat$ is extremal;
\item\label{i:every-epi-is-extremal} every epimorphism in $\cat$ is extremal.
\end{enumerate}
See e.g.\ \cite[Proposition~7.67]{ahs06book} for a proof.

If $\tau$ is an algebraic signature, the extremal epimorphisms in the category $\Alg{\tau}$ of $\tau$-algebras and homomorphisms between them are precisely the surjective homomorphisms. This property is inherited by any full subcategory that is closed in $\Alg{\tau}$ under binary products and subalgebras.
For this reason, in the universal algebra literature, property~\ref{i:every-epi-is-extremal} above is referred to as the \emph{\ES~property} (``Epimorphisms are Surjective'').
We record this notion, along with a strengthening of it:

\begin{definition}
A category $\cat$ is said to have the
\begin{enumerate}
\item[(\ES)]\label{ES} \emph{ES property} if every epimorphism in $\cat$ is an extremal epimorphism;
\item[(\SES)]\label{SES} \emph{strong ES property} if every monomorphism in~$\cat$ is a regular monomorphism.
\end{enumerate}
\end{definition}

Note that $\SES\IMP\ES$ because, in any category, every regular monomorphism is an extremal monomorphism.

\section{Implicit operations}\label{s:implicit-operations-defined}

Fix an arbitrary category $\cat$. The following notion of implicit partial operation is a straightforward generalisation of the one outlined in the Introduction.

\begin{definition}
Let $V\colon \cat\to\Set$ be a functor, and $I$ a set. A \emph{$V$-implicit partial operation of arity $I$} (or \emph{implicit partial operation}, for short) is a span of the form 
\[\begin{tikzcd}
& D \arrow[rightarrowtail]{dl}[swap]{d} \arrow{dr}{f} & \\
V^{I} & & V
\end{tikzcd}\]
where $d\colon D\rightarrowtail V^{I}$ is a subfunctor, called the \emph{domain} of~$f$. If there is no need to specify its domain, we will use the notation $f\colon V^{I} \rightsquigarrow V$. We say that $f$ is \emph{total} if $d$ is the identity transformation. Thus, the $V$-implicit total operations of arity $I$ are precisely the natural transformations $V^{I}\to V$. 
\end{definition}

\begin{remark}\label{rmk:partial-total-ops}
Suppose that $V$ is represented by an object $G$ of $\cat$. The fact that the Yoneda embedding is fully faithful then entails that $V$-implicit total operations of arity $I$ are in bijective correspondence with $\cat$-arrows $G \to \coprod_{I}{G}$. 

For example, if $\cat$ is a quasivariety of algebras and $U$ is the underlying-set functor, the $U$-implicit total operations of arity $I$ are in bijection with the elements of the free algebra on $I$ (which coincides with the $I$-fold copower of the free algebra on one generator). If $I$ is a finite set of cardinality $n$, these can be identified with equivalence classes of terms in $n$ variables in the language of the quasivariety.
\end{remark}

An implicit partial operation $f\colon V^{I}\rightsquigarrow V$ with domain $d\colon D\rightarrowtail V^{I}$ \emph{can be extended to a total one} if there is a natural transformation $g\colon V^{I}\to V$ making the following diagram commute.
\[\begin{tikzcd}
& D \arrow[rightarrowtail]{dl}[swap]{d} \arrow{dr}{f} & \\
V^{I} \arrow{rr}{g} & & V
\end{tikzcd}\]

\section{\ES~and representable implicit partial operations}\label{s:implicit-ES}

\begin{definition}
An implicit partial operation $V^{I} \rightsquigarrow V$ with domain $D\rightarrowtail V^{I}$ is \emph{representable} if the functor $D$ is representable.
\end{definition}

The next result establishes item~\ref{i:ES-char-representable} of Theorem~\ref{t:main}.

\begin{theorem}\label{t:ES-vs-implicit-operations}
Let $\cat$ be a complete category with an extremal projective generator~$G$, and let $V\coloneqq \cat(G,-)\colon\cat\to\Set$. The following statements are equivalent:
\begin{enumerate}[label=(\arabic*),series=implicitops]
\item\label{i:Cat-satisfies-ES} $\cat$ satisfies \ES~(equivalently, $\cat$ is balanced);
\item\label{i:representable-admits-extension} every representable implicit partial operation $V^{I} \rightsquigarrow V$ can be extended to a total one.
\end{enumerate}
\end{theorem}

\begin{proof}
Before starting, let us mention in passing that the implication \ref{i:Cat-satisfies-ES} $\IMP$ \ref{i:representable-admits-extension} holds even when $G$ is not an extremal generator, and the implication \ref{i:representable-admits-extension} $\IMP$ \ref{i:Cat-satisfies-ES} holds even when $G$ is not extremal projective. 

We show that both items~\ref{i:Cat-satisfies-ES} and~\ref{i:representable-admits-extension} are equivalent to the following: 
\begin{enumerate}[label=(\arabic*),resume=implicitops]
\item\label{i:ES-at-copowers-of-G} every epimorphism in $\cat$ whose domain is a small copower of $G$ is an extremal epimorphism.
\end{enumerate}

Clearly, \ref{i:Cat-satisfies-ES} $\IMP$ \ref{i:ES-at-copowers-of-G}. Conversely, suppose that $\delta\colon A\to B$ is an epimorphism in $\cat$. Let~$I$ be the set $\cat(G,A)$, and consider the canonical arrow $\epsilon\colon \coprod_{I}G \to A$. The latter is an (extremal) epimorphism because $G$ is an extremal generator. The composite $\delta\circ \epsilon \colon \coprod_{I}G \to B$ is an epimorphism whose domain is a small copower of $G$ and thus, by assumption, it is an extremal epimorphism. It follows easily that $\delta$ is an extremal epimorphism too.

It remains to show that \ref{i:representable-admits-extension} $\IFF$ \ref{i:ES-at-copowers-of-G}. We start by proving the following fact:
\begin{claim}\label{claim:ext-epi-iff-total}
Let $I$ be a set and let $\delta\colon \coprod_{I}G \to A$ be an epimorphism in $\cat$. Then~$\delta$ is an extremal epimorphism if, and only if, every representable implicit partial operation $V^{I} \rightsquigarrow V$ with domain $\yo(\delta)\colon \yo(A)\rightarrowtail V^{I}$ can be extended to a total one.
\end{claim}
\begin{proof}[Proof of Claim]
Assume that $\delta$ is an extremal epimorphism and let $f\colon \yo(A)\to V$ be a representable implicit partial operation with domain $\yo(\delta)\colon \yo(A)\rightarrowtail V^{I}$. Since the Yoneda embedding is full, there is an arrow $\phi\colon G\to A$ such that $\yo(\phi)=f$. Because $G$ is extremal projective, there exists an arrow $\psi\colon G\to \coprod_{I}{G}$ such that $\delta\circ \psi = \phi$. The implicit total operation $\yo(\psi)\colon V^{I}\to V$ satisfies $\yo(\psi)\circ \yo(\delta) = f$, i.e.\ it extends $f$, as was to be proved.

Conversely, suppose that every representable implicit partial operation $V^{I} \rightsquigarrow V$ with domain $\yo(\delta)\colon \yo(A)\rightarrowtail V^{I}$ can be extended to a total one. That is, for every natural transformation $f\colon \yo(A)\to V$ there exists $g\colon V^{I}\to V$ such that $g\circ \yo(\delta) = f$. Therefore, in $\cat$, for every $\phi\colon G\to A$ there exists $\psi\colon G\to \coprod_{I}G$ such that $\delta\circ \psi = \phi$. 

Let $J$ be the set $\cat(G,A)$, and let $\epsilon\colon \coprod_{J}G \to A$ be the canonical arrow. Since $G$ is an extremal generator, $\epsilon$ is an extremal epimorphism. For each $\phi\colon G \to A$, choose $\psi_{\phi}\colon G\to \coprod_{I}G$ such that $\delta\circ \psi_{\phi} = \phi$. The cocone $\{\psi_{\phi}\mid \phi\in J\}$ induces a unique mediating morphism $\iota\colon \coprod_{J}G \to \coprod_{I}G$, which satisfies $\delta\circ \iota = \epsilon$ by construction. As~$\epsilon$ is an extremal epimorphism, we conclude that $\delta$ is an extremal epimorphism. 
\end{proof}
It follows at once from the latter claim that \ref{i:representable-admits-extension} $\IMP$ \ref{i:ES-at-copowers-of-G}.
For the converse implication, let $f\colon V^{I}\rightsquigarrow V$ be an implicit partial operation whose domain $d\colon D\rightarrowtail V^{I}$ is represented by an object $A$ of $\cat$. 
Thus, the span in the copresheaf category $\CP{\cat}$ displayed on the left-hand side below corresponds to a cospan in $\cat$ as on the right-hand side below.
\begin{equation*}
\begin{tikzcd}
& D \arrow[rightarrowtail]{dl}[swap]{d} \arrow{dr}{f} & \\
V^{I} & & V
\end{tikzcd}
\ \ \ \ \ \ 
\begin{tikzcd}
& A  & \\
\coprod_{I}{G} \arrow{ur}{\delta} & & G \arrow{ul}[swap]{\phi}
\end{tikzcd}
\end{equation*}
The Yoneda embedding $\cat^{\op}\to \CP{\cat}$ reflects monomorphisms because it is faithful, so $\delta$ is an epimorphism in $\cat$. Item~\ref{i:ES-at-copowers-of-G}, combined with Claim~\ref{claim:ext-epi-iff-total}, implies that $f\colon V^{I}\rightsquigarrow V$ can be extended to an implicit total operation, as was to be~proved.
\end{proof}

\begin{remark}\label{rm:profinite-integers}
	In item~\ref{i:representable-admits-extension} of Theorem~\ref{t:ES-vs-implicit-operations}, the requirement that the implicit partial operations be representable cannot be omitted. Consider, e.g., the underlying-set functor $U\colon \AGrp\to \Set$ on the category of Abelian groups. We claim that, even though $\AGrp$ is balanced, not every (unary) implicit partial operation $U\rightsquigarrow U$ can be extended to a total one.
	
	Let $D\colon \AGrp\to \Set$ denote the functor sending an Abelian group to its set of torsion elements.
	While $D$ is a non-representable subfunctor of $U$, it is ind-representable, in the following sense. Write $(\N,\mid)$ for the poset of natural numbers equipped with the divisibility order. For any Abelian group $G$, we have
	\begin{align*}
	D(G) &= \{x\in G\mid nx=0 \text{ for some } n\in \N\} \\
	&= \bigcup_{n\in \N}{\{x\in G\mid nx=0\}} \\
	&\cong \colim_{n\in (\N,\mid)} \AGrp(\Z/n\Z,G).
	\end{align*}
	In fact, $D$ is the directed colimit in $\CP{\AGrp}$ of the representable functors $\yo(\Z/n\Z)$. Therefore,
	\begin{align*}
	\CP{\AGrp}(D,U) &\cong \CP{\AGrp}(\colim_{n\in (\N,\mid)}{\yo(\Z/n\Z)},U) \\
	&\cong \lim_{n\in (\N,\mid)}{\CP{\AGrp}(\yo(\Z/n\Z),U)} \\
	&\cong \lim_{n\in (\N,\mid)}{U(\Z/n\Z)} \\
	&\cong U(\widehat{\Z}),
	\end{align*}
	where $\widehat{\Z}$ is the group of profinite integers. Hence, the implicit partial operations $U\rightsquigarrow U$ with domain $D$ are in one-to-one correspondence with the profinite integers. In contrast, the implicit total operations $U\to U$ are in one-to-one correspondence with the integers (cf.\ Remark~\ref{rmk:partial-total-ops}). Therefore, each element of the remainder $\widehat{\Z}\setminus \Z$, which is an uncountable set, corresponds to an implicit partial operation $U\rightsquigarrow U$ that cannot be extended to a total one.
\end{remark}

In the remainder of this section we show that, if $\cat$ is well-copowered, then representable implicit partial operations can be replaced by continuous ones.

\begin{definition}
An implicit partial operation $V^{I} \rightsquigarrow V$ with domain $D\rightarrowtail V^{I}$ is \emph{continuous} if the functor $D$ preserves limits.
\end{definition}

Continuous implicit partial operations are called \emph{limit-closed} in~\cite{Hebert1998}. The following observation appeared under slightly different assumptions in Lemma~11 of \emph{op.~cit.} It is essentially a consequence of Freyd's General Adjoint Functor Theorem.

\begin{lemma}\label{l:subfunctor-of-representable}
Let $\cat$ be a complete, well-powered and well-copowered category. A subfunctor of a representable functor $\cat \to \Set$ is representable if, and only if, it preserves limits.
\end{lemma}

\begin{proof}
For the non-trivial direction, let $X$ be an object of $\cat$, and $D\colon \cat\to\Set$ a limit-preserving subfunctor of $\yo(X)$. We must prove that $D$ is representable. By Freyd's Representability Theorem (see e.g.\ \cite[Theorem~4.6.15]{Riehl2016book}), it is enough to prove that the following \emph{solution set condition} is verified: there exists a set $S$ of objects of~$\cat$ such that, for any $c\in \cat$ and any $x\in D(c)$, there exist $s\in S$, $y\in D(s)$ and an arrow $f\colon s\to c$ such that $Df(y)=x$. In other words, any arrow $x\colon X\to c$ in $D(c)$ factors through some arrow $y\colon X\to s$ in $D(s)$, for some $s\in S$.

Let $S$ be a representative set for the epimorphic quotients of $X$; such a set exists because $\cat$ is well-copowered. Let $\mathcal{M}$ be the collection of subobjects $m_i\colon u_i \emb c$ such that $x$ factors as $x=m_i\circ y_i$ for some arrow $y_i\colon X\to u_i$ in $D(u_i)$. Note that~$\mathcal{M}$ is a set because $\cat$ is well-powered, and the arrows $y_i$ are necessarily unique.
Since $\cat$ is complete, we can consider the intersection of all subobjects in $\mathcal{M}$, thus obtaining a subobject $m\colon s\emb c$. The object $s$ is the wide pullback of the monomorphisms $m_i$, and the arrows $y_i$ induce a unique mediating arrow $y\colon X\to s$ such that $x = m\circ y$. Because $D$ preserves limits, in particular wide pullbacks, $D(s)$ is the wide pullback of the maps $D(m_i)$. As each $y_i$ belongs to $D(u_i)$, it follows that $y$ belongs to~$D(s)$.

To conclude the proof, we must prove that $y$ is an epimorphism. Suppose that $u,v$ are parallel arrows satisfying $u\circ y = v\circ y$, and let $e\colon E \emb s$ be their equaliser. The universal property of the latter entails the existence of a unique arrow $q\colon X\to E$ such that $y = e\circ q$. Because $D$ preserves equalisers and $y\in D(s)$, it follows that $q\in D(E)$. Since $x = m\circ e \circ q$, the subobject $m\circ e\colon E \emb c$ belongs to $\mathcal{M}$. Clearly, $m\circ e$ is below $m$ in the poset of subobjects of $c$, and so the two subobjects coincide because $m$ is the infimum of $\mathcal{M}$. Therefore, $e$ is an isomorphism and $u=v$.
\end{proof}

Theorem~\ref{t:ES-vs-implicit-operations}, combined with Lemma~\ref{l:subfunctor-of-representable}, yields a proof of item~\ref{i:ES-char} of Theorem~\ref{t:main}.

\begin{corollary}\label{cor:ES-vs-continuous-operations}
Let $\cat$ be a complete well-copowered category with an extremal projective generator~$G$, and let $V\coloneqq \cat(G,-)\colon\cat\to\Set$. The following are equivalent:
\begin{enumerate}[label=(\arabic*),series=implicitops]
\item\label{i:well-copowered-Cat-satisfies-ES} $\cat$ satisfies \ES~(equivalently, $\cat$ is balanced);
\item\label{i:continuous-admits-extension} every continuous implicit partial operation $V^{I} \rightsquigarrow V$ can be extended to a total one.
\end{enumerate}
\end{corollary}

\begin{proof}
The category $\cat$ is well-powered by Lemma~\ref{l:basic-props-complete-gen}\ref{i:well-powered}. Therefore, Lemma~\ref{l:subfunctor-of-representable} applies to show that continuous implicit partial operations coincide with representable ones. The statement then follows from Theorem~\ref{t:ES-vs-implicit-operations}.
\end{proof}

\section{\SES~and principal implicit partial operations}\label{s:implicit-SES}

\begin{definition}\label{def:principal-subfunctor}
Let $F\colon \cat \to\Set$ be a functor, $A$ an object of $\cat$, and $x$ an element of $F(A)$. The \emph{subfunctor $S_{x}$ of $F$ generated by $x$} is defined by
\[
S_{x}(B)\coloneqq \{F(h)(x)\mid h\colon A \to B\}
\]
for every object $B$ of $\cat$. 
A subfunctor of $F$ of the form $S_x$ is called \emph{principal}.
\end{definition}

\begin{remark}
Principal subfunctors of $F$ are precisely the images of natural transformations of the type $\yo(A)\to F$, for $A$ an object of $\cat$. Just observe that an element $x\in F(A)$ induces a morphism $\yo(A)\to F$ which sends $h\in \cat(A,B)$ to $F(h)(x)$. The smallest subfunctor of $F$ through which this morphism factors is precisely $S_{x}$.

If $F$ is of the form $\yo(X)$ for an object $X$ of $\cat$, then $S_{x}(B)$ consists of all arrows $X\to B$ that factor through $x\colon X\to A$. In this case, the principal subfunctors of $F$ are the images of natural transformations $\yo(A) \to\yo(X)$. 
\end{remark}

\begin{definition}
An implicit partial operation $V^{I} \rightsquigarrow V$ with domain $D\rightarrowtail V^{I}$ is \emph{principal} if $D$ is a principal subfunctor of $V^I$.
\end{definition}

The next result proves item~\ref{i:SES-char} of Theorem~\ref{t:main}.

\begin{theorem}\label{t:SES-vs-implicit-product-pres-operations}
Let $\cat$ be a complete category with an extremal projective generator~$G$, and let~$V\coloneqq\cat(G,-)\colon\cat\to\Set$.  The following statements are equivalent:
\begin{enumerate}[label=(\arabic*),series=implicitops-SES]
\item\label{i:Cat-satisfies-SES} $\cat$ satisfies \SES;
\item\label{i:prod-pres-admits-extension} every principal implicit partial operation $V^{I} \rightsquigarrow V$ can be extended to a total one.
\end{enumerate}
\end{theorem}

\begin{proof}
Fix an arbitrary monomorphism $m\colon A\to B$, and consider the class $\Gamma$ of pairs $(\alpha,\beta)$ of parallel arrows with domain $B$ such that $\alpha\circ m = \beta\circ m$. The equaliser of each such pair yields a regular subobject $E_{(\alpha,\beta)} \emb B$. Since $\cat$ is well-powered by Lemma~\ref{l:basic-props-complete-gen}\ref{i:well-powered}, the subobjects of the form $E_{(\alpha,\beta)}$ form a set. As $\cat$ is complete, the intersection 
\[
E\coloneqq \bigwedge{\{E_{(\alpha,\beta)}\mid (\alpha,\beta)\in\Gamma\}}
\] 
exists and is a regular subobject of $B$ (see e.g.\ \cite[\S 14.5.1]{Schubert1972}). Moreover, in the poset of subobjects of $B$, (the equivalence class of) $m$ is below $E$. In fact, $m$ is a regular monomorphism if, and only if, it coincides with $E$. Because $G$ is an extremal generator, this is equivalent to saying that every arrow $G\to E$ factors through $m$. 

It follows that item~\ref{i:Cat-satisfies-SES} is equivalent to the following:
\begin{enumerate}[label=(\arabic*),resume=implicitops-SES]
\item\label{i:SES-at-G} for any monomorphism $m\colon A\emb B$, any arrow $x\colon G \to B$ satisfying
\[
\forall \alpha,\beta \ (\alpha\circ m = \beta\circ m \ \Rightarrow \ \alpha\circ x = \beta\circ x)
\]
must factor through $m$.
\end{enumerate}
It remains to show that items~\ref{i:prod-pres-admits-extension} and~\ref{i:SES-at-G} are equivalent.

\ref{i:prod-pres-admits-extension} $\IMP$ \ref{i:SES-at-G}. Fix a monomorphism $m\colon A\to B$, and suppose that $x\colon G\to B$ satisfies $\alpha\circ x = \beta\circ x$ for every pair of parallel arrows $(\alpha,\beta)$ such that $\alpha\circ m = \beta\circ m$. We must show that $x$ factors through $m$. 
Write $\epsilon_{A}\colon \coprod_{I}{G}\to A$ for the canonical extremal epimorphism, where $I\coloneqq \cat(G,A)$, and let $D$ be the subfunctor of $V^{I}$ generated by $m\circ \epsilon_{A}\colon \coprod_{I}{G}\to B$. Define a natural transformation $\phi\colon D\to V$ as follows. For every object $K$ of $\cat$, the component of $\phi$ at $K$ is
\[
\phi_{K}\colon \{h\circ m\circ \epsilon_{A}\mid h\in \cat(B,K)\} \to \cat(G,K), \ \ \ \ h\circ m\circ \epsilon_{A}\mapsto h\circ x.
\]
To see that the functions $\phi_{K}$ are well-defined, suppose that $h_{0},h_{1}\in \cat(B,K)$ satisfy $h_{0}\circ m\circ \epsilon_{A}=h_{1}\circ m\circ \epsilon_{A}$. Since $\epsilon_{A}$ is an epimorphism, we get $h_{0}\circ m=h_{1}\circ m$, and therefore $h_{0}\circ x=h_{1}\circ x$. Hence, $\phi_{K}$ is well-defined. It is easy to see that the naturality condition is satisfied, yielding a natural transformation $\phi\colon D\to V$. The latter is a principal implicit partial operation $V^{I}\rightsquigarrow V$, and so there exists $\psi\colon V^{I}\to V$ making the following diagram commute.
\[\begin{tikzcd}
& D \arrow{dr}{\phi} & \\
V^{I} \arrow{rr}{\psi} \arrow[hookleftarrow]{ur} & & V
\end{tikzcd}\]
By the Yoneda Lemma, there is $p\colon G\to \coprod_{I}G$ such that $\psi=\yo(p)=-\circ p$. Thus,
\[
x = \phi_{B}(\id\circ m\circ \epsilon_{A}) = \psi_{B}(m\circ \epsilon_{A}) = m\circ \epsilon_{A}\circ p,
\]
showing that $x$ factors through $m$. 

\ref{i:SES-at-G} $\IMP$ \ref{i:prod-pres-admits-extension}. Let $f\colon V^{I}\rightsquigarrow V$ be an implicit partial operation with domain $d\colon D\rightarrowtail V^{I}$, and suppose there is an arrow $y\colon \coprod_{I}G \to B$ that generates $D$. By Lemma~\ref{l:basic-props-complete-gen}\ref{i:factorisations}, $y$ can be decomposed as an extremal epimorphism $e$ followed by a monomorphism $m$, as displayed below.
\[\begin{tikzcd}
\coprod_{I}{G} \arrow{rr}{y} \arrow{dr}[swap]{e} & & B \\
& A \arrow[rightarrowtail]{ur}[swap]{m}
\end{tikzcd}\]
Consider the arrow $x\coloneqq f_{B}(y)\colon G\to B$. We claim that $\alpha\circ x = \beta\circ x$ for every pair of parallel arrows $(\alpha,\beta)$ such that $\alpha\circ m = \beta\circ m$. Note that any pair $(\alpha,\beta)$ satisfying $\alpha\circ m = \beta\circ m$ must also satisfy $\alpha\circ y = \beta\circ y$. Therefore, using the naturality squares
\[\begin{tikzcd}
D(B) \arrow{r}{f_{B}} \arrow{d}[swap]{\gamma\circ -} & \cat(G,B) \arrow{d}{\gamma\circ -} \\
D(C) \arrow{r}{f_{C}} & \cat(G,C)
\end{tikzcd}\]
for $\gamma\in \{\alpha,\beta\}$, we get
\[
\alpha\circ x = f_{C}(\alpha\circ y) = f_{C}(\beta\circ y) = \beta\circ x.
\]
Item~\ref{i:SES-at-G} then entails that $x$ factors through $m$, i.e., there is $w\colon G\to A$ such that $m\circ w = x$. Since $e$ is an extremal epimorphism and $G$ is extremal projective, there exists $z\colon G\to \coprod_{I}{G}$ such that $e\circ z = w$. To see that the total implicit operation $\yo(z)\colon V^{I}\to V$ extends $f$, it suffices to show that $\yo(z)\circ d = f$. Since $D$ is generated by $y$, this reduces to showing that both sides of the equation send $y$ to the same element. In turn, this holds because
\[
\yo(z)_{B}(y) = y\circ z = m\circ e \circ z = m\circ w = x = f_{B}(y).\qedhere
\] 
\end{proof}

\section{Lfp categories with the amalgamation property}\label{s:lfps-with-AP}

In this section, we will show that in any locally finitely presentable category with the amalgamation property, properties \ES~and \SES~are both equivalent to the \emph{weak \ES~property}, as defined below (Theorem~\ref{t:WES-to-SES}). This result will be used in the next section to characterise balancedness in terms of finitary implicit partial operations.

To start with, recall that a category $\cat$ is said to have
\begin{enumerate}
\item[(AP)]\label{AP} the \emph{amalgamation property} if any pair of monomorphisms $\alpha,\beta$ with common domain can be completed to a commutative square as displayed below, with $\gamma, \delta$ monic.
\[\begin{tikzcd}
\phantom{\cdot} \arrow{r}{\alpha} \arrow{d}[swap]{\beta} & \phantom{\cdot} \arrow[dashed]{d}{\delta} \\
\phantom{\cdot} \arrow[dashed]{r}{\gamma} & \phantom{\cdot}
\end{tikzcd}\]
\end{enumerate}

\begin{remark}
If $\cat$ admits pushouts, then it has \AP~if, and only if, monomorphisms are stable under pushouts along monomorphisms. 
\end{remark}

Next, we review the basic definitions concerning locally finitely presentable categories; for a thorough treatment, see~\cite{ar94book}. An object $X$ of a category $\cat$ is \emph{finitely presentable} (respectively, \emph{finitely generated}) if the functor $\cat(X,-)\colon \cat\to\Set$ preserves directed colimits (respectively, directed colimits of monomorphisms).
\begin{definition}
A category $\cat$ is \emph{locally finitely presentable} (\emph{lfp}, for short) if it is cocomplete and admits a set $\sf G$ of finitely presentable objects such that every object of $\cat$ is a directed colimit of objects from $\sf G$.
\end{definition}

Every lfp category is complete and well-copowered (see e.g.\ \cite[Corollary~1.28 and Theorem~1.58]{ar94book}).

\begin{definition}\label{d:WES}
A category $\cat$ has the
\begin{enumerate}
\item[(\WES)]\label{WES} \emph{weak ES property} if every epimorphism between finitely generated objects of $\cat$ is an extremal epimorphism.
\end{enumerate}
\end{definition}

\begin{remark}\label{rm:WES-in-lfp}
In an lfp category, finitely generated objects are closed under extremal images (cf.\ \cite[Proposition~1.69]{ar94book}). Therefore, in that case, \WES~holds precisely when every epi-mono between finitely generated objects is an isomorphism.
\end{remark}

Clearly, $\SES\IMP\ES\IMP\WES$. The next result shows that, in any lfp category with the amalgamation property, ${\WES\IMP\SES}$, and so all three properties are equivalent.
 For prevarieties (classes of algebras closed under isomorphic copies, subalgebras, and Cartesian products) with the amalgamation property, this was established in~\cite[Theorem~1.3]{BMR2017}. 

\begin{theorem}\label{t:WES-to-SES}
In any lfp category with the amalgamation property, ${\WES\IMP\SES}$.
\end{theorem}

We start by proving two lemmas, which will be used in the proof of Theorem~\ref{t:WES-to-SES}. The first one states that, in any category with enough limits and colimits, if \AP~holds, then every monomorphism has a ``coregular factorisation''.

\begin{lemma}\label{l:epi-reg-mono-for-monos}
	Let $\cat$ be a category with equalisers of pairs of monos and pushouts of spans of monos. If $\cat$ satisfies the amalgamation property, then every monomorphism factors as an epimorphism followed by a regular monomorphism.
	\end{lemma}
	
	\begin{proof}
	Let $f$ be a monomorphism in $\cat$. Consider the pushout square
	\[\begin{tikzcd}
	 \arrow{r}{f} \arrow{d}[swap]{f} & {} \arrow{d}{q_{1}} \\
	\arrow{r}[swap]{q_{2}} {} & {} 
	\end{tikzcd}\]
	and let $e$ be the equaliser of $q_{1}$ and $q_{2}$ (which exists because $q_{1}$ and $q_{2}$ are monomorphisms by virtue of the amalgamation property). That is, $e$ is the equaliser of the cokernel pair of~$f$.
	Since $q_1\circ f=q_2\circ f$, there is a unique morphism $s$ such that
	\[
	f=e\circ s.
	\]
	It remains to prove that $s$ is an epimorphism, for then the latter factorisation has the desired properties.

	Note that $s$ is a monomorphism because so is $f$, and thus the cokernel pair of~$s$, as displayed below, exists and consists of monomorphisms $r_1$ and $r_2$.
	\[\begin{tikzcd}
{} \arrow{r}{s} \arrow{d}[swap]{s} & {} \arrow{d}{r_{1}} \arrow[bend left=25]{ddr}[description]{q_{1}\circ e} & {} \\
\arrow{r}{r_{2}} \arrow[bend right=25]{drr}[description]{q_{2}\circ e} {} & {} \arrow[dashed]{dr}{t}  & {} \\
{} & {} & {} 
\end{tikzcd}\]
	Then $s$ is an epimorphism if, and only if, $r_1=r_2$. Since $q_{1}\circ e=q_{2}\circ e$, the universal property of the pushout yields a unique arrow~$t$ making the above diagram commute. Suppose for a moment that $t$ is a monomorphism. Then
	\[
	t\circ r_{1} = q_{1}\circ e = q_{2}\circ e = t\circ r_{2}
	\]
	implies $r_{1}=r_{2}$, and so $s$ is an epimorphism.

	To see that $t$ is monic, consider the diagram below. It consists of four pushout squares, which exist and consist of monomorphisms by the amalgamation property.
	\[\begin{tikzcd}
	{} \arrow{r}{s} \arrow{d}[swap]{s} & {} \arrow{r}{e} \arrow{d}{r_{1}} & {} \arrow{d} \\
	{} \arrow{r}{r_{2}} \arrow{d}[swap]{e} & {} \arrow{r} \arrow{d}  \arrow[dashed]{dr}[description]{t} & {} \arrow{d}  \\
	{} \arrow{r} & {} \arrow{r}  & {} 
	\end{tikzcd}\]
	In the outer square, which is a pushout by the pasting lemma, the left vertical arrow and the top horizontal one coincide with $f$. Hence, the right vertical arrow and the bottom horizontal one can be identified with $q_{1}$ and $q_{2}$, respectively. It follows by definition of the mediating morphism $t$ that the latter coincides with either composite in the bottom right pushout square. Therefore, $t$ is a monomorphism.
	\end{proof}

The second lemma allows us to reduce property \SES~to finitely generated objects.
\begin{lemma}\label{l:SES-to-finitely-generated-objects}
	In a locally finitely presentable category, \SES~holds if, and only if, every monomorphism between finitely generated objects is a regular mono.
	\end{lemma}

	\begin{proof}
	For the non-trivial direction, suppose that $\cat$ is a locally finitely presentable category in which every monomorphism between finitely generated objects is regular. We start by proving the following fact.

	\begin{claim}\label{claim:fg-mono-reg}
	Every monomorphism with finitely generated domain is regular.
	\end{claim}
	\begin{proof}[Proof of Claim]
	Let $m\colon W\to X$ be a monomorphism with $W$ finitely generated. The object $X$ is the colimit of the complete diagram of finitely generated objects that admit a monomorphism into it. This diagram is directed and consists of monomorphisms, hence of regular monomorphisms. Every colimit cocone for this diagram consists of regular monomorphisms by \cite[Proposition~1.62(i)]{ar94book}; in particular, this applies to the canonical colimit cocone consisting of all monomorphisms into~$X$ with finitely generated domain. Since $m$ belongs to this cocone, we conclude that it is a regular monomorphism.
	\end{proof}
	
	Now, let $f\colon X\to Y$ be any monomorphism in $\cat$. Reasoning as in the proof of the previous claim, we can write $X$ as the colimit of a directed diagram of regular monomorphisms between finitely generated objects. The arrow $f$ induces a compatible cocone of monomorphisms with vertex $Y$ on this diagram. In fact, by Claim~\ref{claim:fg-mono-reg}, this cocone consists entirely of regular monomorphisms.  It follows from \cite[Proposition~1.62(ii)]{ar94book} that the unique mediating morphism $X\to Y$, which coincides with $f$, is a regular monomorphism.
	\end{proof}

Finally, we are in a position to prove Theorem~\ref{t:WES-to-SES}.

\begin{proof}[Proof of Theorem~\ref{t:WES-to-SES}]
By Lemma~\ref{l:SES-to-finitely-generated-objects}, it suffices to show that every monomorphism with finitely generated domain is regular.
Let $f\colon A\to B$ be a monomorphism with~$A$ finitely generated. By Lemma~\ref{l:epi-reg-mono-for-monos}, $f$ factors as an epimorphism followed by a regular monomorphism, say $f=e\circ s$ with $s\colon A \to C$. Note that $s$ is monic because so is $f$. We claim that the epi-mono $s$ is an isomorphism, and so $f$ is regular.

\begin{claim}
	Suppose that whenever $s$ factors as $A\to D\emb C$ with $D$ finitely generated, then $A\to D$ is an isomorphism. Then $s$ is an isomorphism.
	\end{claim}
	\begin{proof}[Proof of Claim]
	Let $P$ be the poset of finitely generated subobjects of $C$. The assumption of the claim means that $s$ is maximal in $P$. The diagram~$\Gamma$ of finitely generated subobjects of $C$ is directed (cf.\ e.g.\ the proof of \cite[Theorem~1.70]{ar94book}), so $P$ is directed. Therefore, $s$ is the maximum of $P$. Since $C$ is the colimit of the diagram $\Gamma$, it follows that $s$ is an isomorphism.
	\end{proof}
	With the aim of applying the previous claim, consider a commutative triangle
	\[\begin{tikzcd}
	A \arrow[rightarrowtail]{rr}{s} \arrow{dr}[swap]{n} & & C \\
	& D \arrow[rightarrowtail]{ur}[swap]{k} &
	\end{tikzcd}\]
	with $D$ finitely generated. We must prove that $n$ is an isomorphism. Note that $n$ is a monomorphism because so is $s$. Therefore, because $\cat$ satisfies \WES, it suffices to prove that $n$ is epic (cf.\ Remark~\ref{rm:WES-in-lfp}), i.e.\ its cokernel pair $(r_{0},r_{1})$ satisfies $r_{0}=r_{1}$.
	Consider the following four pushout squares.
	\[\begin{tikzcd}
	{} \arrow{r}{n} \arrow{d}[swap]{n} & {} \arrow{r}{k} \arrow{d}{r_{0}} & {} \arrow{d}{r'_{0}} \\
	{} \arrow{r}[swap]{r_{1}} \arrow{d}[swap]{k} & {} \arrow{r}[swap]{\rho_{0}} \arrow{d}{\rho_{1}}  & {} \arrow{d}{\rho'_{1}}  \\
	{} \arrow{r}[swap]{r'_{1}} & {} \arrow{r}[swap]{\rho'_{0}}  & {} 
	\end{tikzcd}\]
	Since $s=k\circ n$ is epic its cokernel pair is trivial, i.e.\ $\rho'_{1}\circ r'_{0} = \rho'_{0}\circ r'_{1}$. It follows that 
	\[
	\rho'_{1}\circ r'_{0}\circ k = \rho'_{0}\circ r'_{1}\circ k
	\]
	and so 
	\[ 
	\rho'_{1}\circ \rho_{0}\circ r_{0} = \rho'_{0}\circ \rho_{1}\circ r_{1}.
	\]
	But $\rho'_{1}\circ \rho_{0}=\rho'_{0}\circ \rho_{1}$ is monic by the amalgamation property, therefore $r_{0}=r_{1}$.
\end{proof}

\section{Finitary implicit partial operations}\label{s:finitary-implicit-partial-operations}

The implicit partial operations featured in Theorem~\ref{t:main} have arbitrary (unbounded) arity. Theorem~\ref{t:SES-vs-implicit-operations-with-amalgam} below shows that if, in addition, $\cat$ is locally finitely presentable and satisfies the amalgamation property, then we can restrict ourselves to \emph{finitary} implicit partial operations $V^{I} \rightsquigarrow V$, meaning that $I$ is a finite set. These are denoted by $V^{n} \rightsquigarrow V$. This result relies on the fact that properties \WES, \ES, and \SES~are equivalent in lfp categories with the amalgamation property (Theorem~\ref{t:WES-to-SES}).

\begin{theorem}\label{t:SES-vs-implicit-operations-with-amalgam}
Let $\cat$ be an lfp category satisfying \AP~and admitting an extremal projective generator~$G$, and let $V\coloneqq \cat(G,-)$. The following are equivalent:
\begin{enumerate}[label=(\arabic*)]
\item\label{i:Cat-satisfies-SES-with-AP} $\cat$ satisfies \ES~(equivalently, by Theorem~\ref{t:WES-to-SES}, \WES~or \SES);
\item\label{i:continuous-finitary-admits-extension} every continuous finitary implicit partial operation $V^{n} \rightsquigarrow V$ can be extended to a total one.
\end{enumerate}
\end{theorem}

\begin{proof}
The implication \ref{i:Cat-satisfies-SES-with-AP} $\IMP$ \ref{i:continuous-finitary-admits-extension} is an immediate consequence of Theorem~\ref{t:ES-vs-implicit-operations}.

For the converse implication, it suffices to show that every epimorphism $f\colon A\to B$ with finitely generated domain is an extremal epimorphism, as this entails \WES.
Suppose for a moment that there exists an epimorphism $g\colon \coprod_{n}G \to A$ for some finite set~$n$. Claim~\ref{claim:ext-epi-iff-total}, together with item~\ref{i:continuous-finitary-admits-extension}, entails that the epimorphism $f\circ g$ is an extremal epimorphism. Hence, $f$ is an extremal epimorphism too. 

To construct an epimorphism $g\colon \coprod_{n}G \to A$, we proceed as follows. Let $I$ be the set $\cat(G, A)$, and let $F$ be the set of all finite subsets of $I$. For each $J\in F$, write
\[
\epsilon_{J}\colon \coprod_{J}G\to A
\]
for the composite of the canonical extremal epimorphism $\epsilon \colon \coprod_{I}G \to A$ with the obvious ``restriction'' arrow $\coprod_{J}G \to \coprod_{I}G$ induced by the inclusion $J\subseteq I$. In view of Lemma~\ref{l:basic-props-complete-gen}\ref{i:factorisations}, we can decompose $\epsilon_{J}$ as an extremal epimorphism $e_{J}$ followed by a monomorphism $m_{J}$, as displayed below.
\[\begin{tikzcd}
\coprod_{J}G \arrow{rr}{\epsilon_{J}} \arrow{dr}[swap]{e_{J}} & & A \\
& M_{J} \arrow{ur}[swap]{m_{J}} & 
\end{tikzcd}\]
Consider the directed diagram $\Gamma$ in $\cat$ consisting of the objects of the form $M_{J}$, together with the arrows $h\colon M_{J}\to M_{K}$ such that $m_{K}\circ h = m_{J}$. As $m_{K}$ is a monomorphism, there is at most one such arrow $h$, and it is a monomorphism. If~$M$ denotes the colimit of $\Gamma$, the cocone $\{m_{J}\mid J\in F\}$ induces a mediating arrow
\[
m\colon M \to A.
\]

We claim that $m$ is an isomorphism. Since $m$ is a monomorphism by \cite[Proposition~1.62]{ar94book}, it suffices to show that it is an extremal epimorphism. Note that $\coprod_{I}G$ is the colimit of the diagram whose objects are of the form $\coprod_{J}G$, for $J\in F$, and whose morphisms $\coprod_{J}G \to \coprod_{K}G$ are those induced by inclusions $J\subseteq K$. It follows easily, using the universal property of the colimit $M$, that $\epsilon\colon \coprod_{I}G \to A$ factors through $m$. As $\epsilon$ is an extremal epimorphism, so is $m$. 

We have thus proved that $A$ is the colimit of the directed diagram of monomorphisms $\Gamma$. Because $A$ is finitely generated, the identity arrow $A\to A$ factors through a colimit map $m_{J}\colon M_{J}\rightarrowtail A$, for some $J\in F$, and so $m_{J}$ is an isomorphism. Therefore, $\epsilon_{J}\colon \coprod_{J}G\to A$ is an extremal epimorphism.
\end{proof}

\bibliographystyle{alphaurl-shortnames}

\end{document}

%% file: macros.tex
\newcommand{\N}{\mathbb{N}}				
\newcommand{\Z}{\mathbb{Z}}				

\renewcommand{\phi}{\varphi}
\renewcommand{\epsilon}{\varepsilon}

\newcommand{\cat}{{\sf C}}								
\newcommand{\Set}{{\sf Set}}							
\newcommand{\Top}{{\sf Top}}							
\newcommand{\Frm}{{\sf Frm}}							
\newcommand{\CABA}{{\sf CABA}}							
\newcommand{\CH}{{\sf CH}}							    
\newcommand{\Stone}{{\sf Stone}}					    
\newcommand{\CSLat}{{\sf CSLat}}						
\newcommand{\M}{{\sf M}}							    
\newcommand{\AGrp}{{\sf AbGrp}}							
\newcommand{\Alg}[1]{{\sf Alg}(#1)}						
\newcommand{\K}{{\sf K}}								
\newcommand{\CP}[1]{{#1}^{\text{\tiny$\vee$}}}			
\newcommand{\BA}{{\sf BA}}								

\newcommand{\yo}{\mathcal{Y}}

\newcommand{\id}{\mathrm{id}}								
\newcommand{\emb}{\rightarrowtail}							

\newcommand{\op}{{\mathrm{op}}}						
\newcommand{\colim}{\operatornamewithlimits{colim}}		

\newcommand{\IMP}{\Rightarrow}						
\newcommand{\IFF}{\Leftrightarrow}						

\newcommand{\ES}{\textnormal{ES}}
\newcommand{\WES}{\textnormal{WES}}
\newcommand{\SES}{\textnormal{SES}}
\newcommand{\AP}{\textnormal{AP}}

%% file: implicit_operations_v1.bbl
\begin{thebibliography}{CKM26}

\bibitem[AHS06]{ahs06book}
J.~Ad{\'a}mek, H.~Herrlich, and G.~Strecker.
\newblock {\em {Abstract and concrete categories: The joy of cats}}, volume~17.
\newblock Reprints in Theory and Applications of Categories, 2006.
\newblock Originally published by John Wiley and Sons, New York, 1990.

\bibitem[AR94]{ar94book}
J.~Ad{\'a}mek and J.~Rosick{\'y}.
\newblock {\em {Locally Presentable and Accessible Categories}}, volume 189 of
  {\em London Mathematical Society Lecture Notes Series}.
\newblock Cambridge University Press, 1994.
\newblock \href {https://doi.org/10.1017/CBO9780511600579}
  {\path{doi:10.1017/CBO9780511600579}}.

\bibitem[AT05]{AT2005}
J.~Ad\'amek and V.~Trnkov{\'a}.
\newblock Birkhoff's variety theorem with and without free algebras.
\newblock {\em Theory Appl. Categ.}, 14(18):424--450, 2005.
\newblock \href {https://doi.org/10.70930/tac/kxsbnt6z}
  {\path{doi:10.70930/tac/kxsbnt6z}}.

\bibitem[Bac74]{Bacsich74}
P.~D. Bacsich.
\newblock Model theory of epimorphisms.
\newblock {\em Canad. Math. Bull.}, 17(4):471--477, 1974.
\newblock \href {https://doi.org/10.4153/CMB-1974-083-6}
  {\path{doi:10.4153/CMB-1974-083-6}}.

\bibitem[BMR17]{BMR2017}
G.~Bezhanishvili, T.~Moraschini, and J.~G. Raftery.
\newblock Epimorphisms in varieties of residuated structures.
\newblock {\em Journal of Algebra}, 492:185--211, 2017.
\newblock \href {https://doi.org/10.1016/j.jalgebra.2017.08.023}
  {\path{doi:10.1016/j.jalgebra.2017.08.023}}.

\bibitem[Bor94]{Borceux1}
F.~Borceux.
\newblock {\em Handbook of Categorical Algebra: Volume 1, Basic Category
  Theory}.
\newblock Encyclopedia of Mathematics and its Applications. Cambridge
  University Press, 1994.

\bibitem[BP95]{BP1995}
M.~Barr and M.~C. Pedicchio.
\newblock {${\rm Top}^{\rm op}$} is a quasi-variety.
\newblock {\em Cahiers Topologie G\'eom. Diff\'erentielle Cat\'eg.},
  36(1):3--10, 1995.

\bibitem[CKM26]{CKM2026}
L.~Carai, M.~Kurtzhals, and T.~Moraschini.
\newblock The theory of implicit operations, 2026.
\newblock \href {https://arxiv.org/abs/2512.14326} {\path{arXiv:2512.14326}}.

\bibitem[DLR26]{DLR2026}
I.~Di~Liberti and L.~Reggio.
\newblock {B}eth companions of finitary essentially algebraic theories, 2026.
\newblock \href {https://arxiv.org/abs/2609.03601} {\path{arXiv:2609.03601}}.

\bibitem[Faw86]{Fawcett1986}
B.~Fawcett.
\newblock A categorical characterization of the four colour theorem.
\newblock {\em Canad. Math. Bull.}, 29(4):426--431, 1986.
\newblock \href {https://doi.org/10.4153/CMB-1986-067-8}
  {\path{doi:10.4153/CMB-1986-067-8}}.

\bibitem[Gai64]{Gaifman1964}
H.~Gaifman.
\newblock Infinite {B}oolean polynomials. {I}.
\newblock {\em Fund. Math.}, 54:229--250, 1964.
\newblock \href {https://doi.org/10.4064/fm-54-3-229-250}
  {\path{doi:10.4064/fm-54-3-229-250}}.

\bibitem[GU71]{GabrielUlmer1971}
P.~Gabriel and F.~Ulmer.
\newblock {\em Lokal pr{\"a}sentierbare Kategorien}, volume 221 of {\em Lecture
  Notes in Mathematics}.
\newblock Springer-Verlag, Berlin and Heidelberg, 1971.
\newblock \href {https://doi.org/10.1007/BFb0059396}
  {\path{doi:10.1007/BFb0059396}}.

\bibitem[Hal64]{Hales1964}
A.~W. Hales.
\newblock On the non-existence of free complete {B}oolean algebras.
\newblock {\em Fund. Math.}, 54:45--66, 1964.
\newblock \href {https://doi.org/10.4064/fm-54-1-45-66}
  {\path{doi:10.4064/fm-54-1-45-66}}.

\bibitem[H{\'e}b90]{Hebert1990}
M.~H{\'e}bert.
\newblock Sur le rang et la d\'efinissabilit\'e{} des op\'erations implicites
  dans les classes d'alg\`ebres.
\newblock {\em Colloq. Math.}, 58(2):175--188, 1990.
\newblock \href {https://doi.org/10.4064/cm-58-2-175-188}
  {\path{doi:10.4064/cm-58-2-175-188}}.

\bibitem[H{\'e}b93]{Hebert1993}
M.~H{\'e}bert.
\newblock Sur les op\'erations partielles implicites et leur relation avec la
  surjectivit\'e{} des \'epimorphismes.
\newblock {\em Canad. J. Math.}, 45(3):554--575, 1993.
\newblock \href {https://doi.org/10.4153/CJM-1993-029-3}
  {\path{doi:10.4153/CJM-1993-029-3}}.

\bibitem[H{\'e}b98]{Hebert1998}
M.~H{\'e}bert.
\newblock On generation and implicit partial operations in locally presentable
  categories.
\newblock {\em Appl. Categ. Structures}, 6(4):473--488, 1998.
\newblock \href {https://doi.org/10.1023/A:1008653513618}
  {\path{doi:10.1023/A:1008653513618}}.

\bibitem[HI67]{HI1967}
J.~M. Howie and J.~R. Isbell.
\newblock Epimorphisms and dominions. {II}.
\newblock {\em J. Algebra}, 6:7--21, 1967.
\newblock \href {https://doi.org/10.1016/0021-8693(67)90010-5}
  {\path{doi:10.1016/0021-8693(67)90010-5}}.

\bibitem[Hod80]{Hodges1980}
W.~Hodges.
\newblock Functorial uniform reducibility.
\newblock {\em Fundamenta Mathematicae}, 108(2):77--81, 1980.
\newblock \href {https://doi.org/10.4064/fm-108-2-77-81}
  {\path{doi:10.4064/fm-108-2-77-81}}.

\bibitem[Isb66]{Isbell1966}
J.~R. Isbell.
\newblock Epimorphisms and dominions.
\newblock In {\em Proc. {C}onf. {C}ategorical {A}lgebra ({L}a {J}olla,
  {C}alif., 1965)}, pages 232--246. Springer-Verlag New York, Inc., New York,
  1966.
\newblock \href {https://doi.org/10.1007/978-3-642-99902-4_9}
  {\path{doi:10.1007/978-3-642-99902-4_9}}.

\bibitem[Isb73]{Isbell1973}
J.~R. Isbell.
\newblock Functorial implicit operations.
\newblock {\em Israel Journal of Mathematics}, 15(2):185--188, 1973.
\newblock \href {https://doi.org/10.1007/BF02764604}
  {\path{doi:10.1007/BF02764604}}.

\bibitem[Isb83]{Isbell1983}
J.~Isbell.
\newblock Generic algebras.
\newblock {\em Trans. Amer. Math. Soc.}, 275(2):497--510, 1983.
\newblock \href {https://doi.org/10.2307/1999036} {\path{doi:10.2307/1999036}}.

\bibitem[Joh82]{Johnstone1982}
P.~T. Johnstone.
\newblock {\em Stone spaces}, volume~3 of {\em Cambridge Studies in Advanced
  Mathematics}.
\newblock Cambridge University Press, Cambridge, 1982.

\bibitem[Joh02]{Elephant1}
P.~T. Johnstone.
\newblock {\em Sketches of an elephant: a topos theory compendium. {V}ol. 1},
  volume~43 of {\em Oxford Logic Guides}.
\newblock The Clarendon Press, Oxford University Press, New York, 2002.

\bibitem[KG71]{KG1971}
J.~F. Kennison and D.~Gildenhuys.
\newblock Equational completion, model induced triples and pro-objects.
\newblock {\em J. Pure Appl. Algebra}, 1(4):317--346, 1971.
\newblock \href {https://doi.org/10.1016/0022-4049(71)90001-6}
  {\path{doi:10.1016/0022-4049(71)90001-6}}.

\bibitem[Law63]{Lawvere1963}
F.~W. Lawvere.
\newblock Functorial semantics of algebraic theories.
\newblock {\em Proc. Nat. Acad. Sci. U.S.A.}, 50:869--872, 1963.
\newblock \href {https://doi.org/10.1073/pnas.50.5.869}
  {\path{doi:10.1073/pnas.50.5.869}}.

\bibitem[Lin66]{Linton1966}
F.~E.~J. Linton.
\newblock Some aspects of equational categories.
\newblock In {\em Proc. {C}onf. {C}ategorical {A}lgebra ({L}a {J}olla,
  {C}alif., 1965)}, pages 84--94. Springer-Verlag New York, Inc., New York,
  1966.
\newblock \href {https://doi.org/10.1007/978-3-642-99902-4_3}
  {\path{doi:10.1007/978-3-642-99902-4_3}}.

\bibitem[Mal66]{Maltsev1966}
A.~I. Maltsev.
\newblock Several remarks on quasivarieties of algebraic systems.
\newblock {\em Algebra i Logika Sem.}, 5(3):3--9, 1966.

\bibitem[Man76]{Manes1976}
E.~G. Manes.
\newblock {\em Algebraic theories}.
\newblock Springer-Verlag, New York-Heidelberg, 1976.
\newblock Graduate Texts in Mathematics, No. 26.
\newblock \href {https://doi.org/10.1007/978-1-4612-9860-1}
  {\path{doi:10.1007/978-1-4612-9860-1}}.

\bibitem[PP12]{PP2012}
J.~Picado and A.~Pultr.
\newblock {\em Frames and locales. Topology without points}.
\newblock Frontiers in Mathematics. Birkh\"auser/Springer Basel AG, Basel,
  2012.
\newblock \href {https://doi.org/10.1007/978-3-0348-0154-6}
  {\path{doi:10.1007/978-3-0348-0154-6}}.

\bibitem[Rei82]{Reiterman1982}
J.~Reiterman.
\newblock The {B}irkhoff theorem for finite algebras.
\newblock {\em Algebra Universalis}, 14(1):1--10, 1982.
\newblock \href {https://doi.org/10.1007/BF02483902}
  {\path{doi:10.1007/BF02483902}}.

\bibitem[Rie16]{Riehl2016book}
E.~Riehl.
\newblock {\em Category theory in context}.
\newblock Aurora Dover Modern Math Originals. Dover Publications, Inc.,
  Mineola, NY, 2016.

\bibitem[RP81]{RP1981}
J.~Rosick\'y and L.~Pol\'ak.
\newblock Implicit operations on finite algebras.
\newblock In {\em Finite algebra and multiple-valued logic ({S}zeged, 1979)},
  volume~28 of {\em Colloq. Math. Soc. J\'anos Bolyai}, pages 653--668.
  North-Holland, Amsterdam-New York, 1981.

\bibitem[Sch72]{Schubert1972}
H.~Schubert.
\newblock {\em Categories}.
\newblock Springer-Verlag, New York-Heidelberg, 1972.
\newblock Translated from the German by Eva Gray.

\bibitem[Sto36]{Stone1936}
M.~H. Stone.
\newblock The theory of representations for {B}oolean algebras.
\newblock {\em Trans. Amer. Math. Soc.}, 40(1):37--111, 1936.
\newblock \href {https://doi.org/10.1090/S0002-9947-1936-1501865-8}
  {\path{doi:10.1090/S0002-9947-1936-1501865-8}}.

\bibitem[Sto68]{Storrer1968}
H.~H. Storrer.
\newblock Epimorphismen von kommutativen {R}ingen.
\newblock {\em Comment. Math. Helv.}, 43:378--401, 1968.
\newblock \href {https://doi.org/10.1007/BF02564404}
  {\path{doi:10.1007/BF02564404}}.

\end{thebibliography}
